\documentclass[12pt,twoside,reqno]{amsart}
\usepackage{amsmath, amsfonts, amsthm, amssymb,amscd, graphicx, amscd}
\usepackage{float,epsf}
\usepackage[english]{babel}
\usepackage{enumerate}
\usepackage{tikz}
\usepackage{enumerate}
\usepackage{etaremune}
\usepackage{srcltx}
\usepackage{geometry}
\usepackage{verbatim}
\usepackage{mathrsfs}
\usepackage{hyperref}
\usepackage{enumitem}
\setlist[enumerate,1]{label=\arabic*.}
\usepackage{esint}
\usepackage{pgfplots}

\newcommand{\R}{\mathbb{R}}

\newcommand{\tn}{\textnormal}

\newcommand{\eps}{\varepsilon}

 \newcommand{\supp}{\text{\rm supp}\,}

\newcommand{\grad}{\nabla}
 \renewcommand{\supp}{\text{\rm supp}\,}

\newtheorem{theorem}{Theorem}[section]
\newtheorem{lemma}[theorem]{Lemma}

\newtheorem{definition}[theorem]{Definition}

\newtheorem{proposition}[theorem]{Proposition}
\newtheorem{corollary}[theorem]{Corollary}
\newtheorem{remark}[theorem]{Remark}
\newtheorem{example}[theorem]{Example}

\numberwithin{equation}{section}
\numberwithin{figure}{section}

\newcommand{\norm}[1]{\left\|#1\right\|}

\newcommand{\FF}{{\boldsymbol F}}

\renewcommand{\div}{\mathrm{div}} 

\renewcommand{\div}{\text{\sl div}\,} 

\def\rightangle{\vcenter{\hsize5.5pt
    \hbox to5.5pt{\vrule height7pt\hfill}
    \hrule}}
\def\rtangle{\mathrel{\rightangle}}
\def\intave#1{\int_{#1}\hbox{\llap{$\raise2.3pt\hbox{\vrule
height.9pt width7pt}\phantom{\scriptstyle{#1}}\mkern-2mu$}}}

\begin{document}
\title[The free Boundary in a Long-Range Segregation Model]{The free Boundary in a higher-dimensional Long-Range Segregation Model}

\author{Howen Chuah}
\address{Howen Chuah,
Department of Mathematics, Purdue University,
 West Lafayette, IN 47907-2067, USA}
\email{hchuah@purdue.edu}


\author{Monica Torres}
\address{Monica Torres,
Department of Mathematics, Purdue University,
 West Lafayette, IN 47907-2067, USA}
 \email{torres@math.purdue.edu}


\begin{abstract}
We consider a system of elliptic equations, depending on a small parameter  $\eps$,  that models long-range segregation of populations. The diffusion is governed by the Laplacian. This system was previously investigated by Caffarelli, Patrizi, and Quitalo in \cite{CL2} as a model in population dynamics, and they established the regularity of the free boundary in two dimensions. In this paper we study the free boundary in the higher dimensional case. We extend the concept of angles and asymptotic cones to higher dimensions, and give a characterization of regular and singular points in terms of their densities and angles. We obtain a structure result of the free boundary and show that, if the angles at the singular points are away from $\frac{n\omega_n}{2}$, the regular set is open in the free boundary and locally a $C^1$ manifold of dimension $n-1$. We also show that, if the supports of the populations are convex, they are convex polytopes. {A weak form of the equality of angles for the convex configuration is also derived.}
\end{abstract}
\maketitle

\vspace{-6mm}
\begingroup\hypersetup{hidelinks} 
\endgroup
\newcommand{\sizeofboundary}{R}

\section{Introduction}

In this paper, we consider the following system of elliptic equations 

\begin{equation}
\label{main_problems}
\begin{cases} \Delta u^{\varepsilon}_{i}= \frac{1}{\eps^2} u^{\eps}_i  \sum_{j \neq i}  H_R(u^{\eps}_j)(x)\quad & \text{ in } \Omega,\\
        u^{\eps}_i= f_i & \text{ on } (\partial \Omega)_{\leq {\sizeofboundary}}, \\
        
        u^{\eps}_i \geq 0 & \text{ in } \Omega \cup (\partial \Omega)_{\leq \sizeofboundary},
       \end{cases}
\end{equation}
for  $i=1,\ldots, K$, 
where $\Omega$ is a bounded Lipschitz domain in  $\R^n$, $\eps >0$, $0 < \sizeofboundary \leq 1$. The boundary neighborhood is defined as 
\begin{equation*}
    (\partial \Omega)_{\leq \sizeofboundary}:= \{x \in \Omega^c: d(x, \partial \Omega)\leq \sizeofboundary\} ,
\end{equation*}
where $d(x,\partial \Omega) := \inf_{y \in \partial \Omega}|x-y|$ denotes the distance of  $x$ to $\partial \Omega$. 
The boundary data are nonnegative H{\"o}lder continuous functions
with supports separated by at least distance $\sizeofboundary$ (see assumptions \eqref{fiassumption}-\eqref{f_i-f_j_disjointsupport}).

Each equation in the system is coupled to the others through a nonlocal zero-order interaction  term
$H_R(u_j^\eps)$, which depends  on the parameter $R$. 
We consider two types of nonlocal interactions: 


\begin{equation}
\label{H1}
H_R(w)(x)= \fint_{{B}_\sizeofboundary(x)}w(y)\,dy, 
\end{equation}
and 
\begin{equation}
\label{H2}
 H_R(w)(x)= \sup_{{B}_R(x)} w, 
\end{equation}
where $w \geq 0$. We assume that the boundary data satisfy  
\begin{equation}    
\label{fiassumption}
        f_i : (\partial \Omega)_{\leq R} \rightarrow \mathbb{R}, f_i \geq 0, f_i \neq 0, \quad \textnormal{$f_i$ is H\"older continuous,} 
    \end{equation}
and that there is a constant $c > 0$ such that, for any $x \in \partial \Omega \cap \text{supp } f_i$,
\begin{equation}
\label{size_of_ball_intersect_support}
|{B}_r(x) \cap \text {supp }f_i| \geq c|{B}_r(x)|,
\end{equation}
and
       \begin{equation}\label{f_i-f_j_disjointsupport}       
        d(\tn{supp} \, f_i, \tn{supp} \, f_j) \geq \sizeofboundary \textnormal{ for all } i \neq j, \\
        \end{equation}
It is also assumed that each $\supp f_i \cap \partial\Omega$ has finitely many connected components.

The existence of positive solutions $(u_1^\eps,\ldots,u_K^\eps)$ of the system \eqref{main_problems} was proved in \cite{CL2}. They also showed that solutions converge to a limit configuration $(u_1,\ldots, u_K)$ as $\eps \to 0$, where the supports of the populations $u_i$  are mutually disjoint and separated from each other by distance $\sizeofboundary$ (see Figure \ref{generalpicture}). The regularity of the free boundary for $n=2$ was also established. Uniqueness of solutions to system \eqref{main_problems} was proved in \cite{Bozorgnia} with  $H_R$ defined as in \eqref{H1}. For system \eqref{main_problems}, one of the main challenges is that classical techniques of free boundaries (i.e., monotonicity formulas) can not be used. For $n=2$, the techniques developed in \cite{CL2} rely on the concepts of {\it asymptotic cone} and {\it asymptotic angle}. At a singular point, the angle is measured as the intersection of asymptotic cones. These concepts do not carry over to higher dimensions in a straightforward manner. In this paper, we replace the intersection of cones with the intersection of half spheres and the {\it angle} with the Hausdorff measure of such intersection. This generalization reduces to the classical asymptotic cone, introduced in \cite{CL2}, for $n=2$. We also establish a characterization of regular and singular points, in any dimension, in terms of angles and densities (see Theorem \ref{Regular_points_Singular_points_in_Higher_Dimensions}).

For the case of $n=2$, it was shown that the singular points are isolated and the regular set is locally $C^1$. It was also shown, under additional conditions, that the free boundary is Lipschitz. A result on the equality of angles is also established (i.e., if two free boundary points in different populations are at distance $\sizeofboundary$, their corresponding angles are equal). The proof is based on the  geometry of $\R^2$ and it does not have easy counterparts in higher dimensions. Another tool for $n=2$ is the construction of harmonic functions (i.e., barrier functions) on cones vanishing on the boundary of these cones, but these constructions depend heavily on the geometry of $\R^2$. In this paper, we study the regularity of the free boundary in higher dimensions. We show that, if the angles at the singular points are bounded away from $\frac{n\omega_n}{2}$, the regular set is open and locally $C^1$ (see Theorem \ref{Combined_Results_on_Regular_Set}). We also show that, under a convexity condition, the free boundary consist of finitely many hyperplanes of dimension $n-1$, and the angles of the singular points are bounded above by $\frac{n\omega_n}{3}$. We also prove that, if there are two free boundary points that are at distance $\sizeofboundary$ apart form each other, either they are both regular or they are both singular. A symmetry condition on the free boundary is also derived (see Theorem \ref{Convex}).



System \eqref{main_problems} is an example of \eqref{modelproblem}, the Gause-Lotka-Volterra system in population dynamics, that models coexistence of species that live in the same territory, diffuse, and compete for limited resources.
 
\begin{equation}\label{modelproblem}L_i(u^\eps _i)=\frac{u^\eps_i}{\eps^2}F(u^\eps_1,\ldots,u^\eps_K),\end{equation}
in some domain $\Omega$, where $u^\eps_i$ is a positive function representing the density of the $i$-th species, $L_i$ encodes  the diffusion of  $u^\eps_i$, and   $u_i^\eps F(u^\eps_1,\ldots,u^\eps_K)/\eps^2$ models  the  attrition of the species $i$ due to competition with the others. The interaction functional  $F$ is strictly positive whenever the supports of two or more species overlap.
The smaller the parameter $\eps$, the stronger the competition among species. 
In the limit as $\eps\to0^+$ the high  competition  forces the species  to   segregate, meaning  $u_iu_j=0$ for $j\neq i$. 

Another example of system \eqref{modelproblem} is given by
\begin{equation}
\label{adjacentsegregation}\Delta u^\eps _i=\frac{1}{\eps^2}\sum_{j\neq i}u^\eps_i u_j^\eps. \end{equation}

The existence of positive solutions to \eqref{adjacentsegregation} was initially investigated by Dancer and Du \cite{DancerDu1, DancerDu2} in the case of three species. 
Convergence to a segregated limit configuration as  $\eps\to0^+$ was later proven by Dancer, Hilhorst, Mimura, and Peletier \cite{DancerHilMimPel}. 
 More general classes of linear competitive systems, including \eqref{adjacentsegregation} as a special case, have been studied by Conti, Terracini, and Verzini \cite{Conti4, Conti3, Conti5}. %
 We also refer to \cite{CL7, Conti6} for related optimal partition problems involving the first eigenvalue of the Laplace operator.  
 The geometric properties of the free boundaries $\partial\{u_i>0\}\cap\Omega$ for the system \eqref{adjacentsegregation} have been investigated by Caffarelli, Karakhanyan and Lin \cite{CL5} (see also \cite{CL4}). It was shown that the free boundary splits into two parts: a regular set,  which is a locally  analytic surface, and a singular set, which is a closed set of Hausdorff dimension at most $n-2$. Singular points occur where the boundaries of three or more connected components of the supports intersect. See \cite{TavarezTerracini} for similar results applied to a broader class of systems.
{

The system \eqref{adjacentsegregation}, when the Laplace operator is replaced by the fully nonlinear negative Pucci operator, has been studied by Quitalo \cite{V}.
In \cite{V}, existence of positive solutions and convergence to  a limiting segregated configuration is established. In the case of two populations, Caffarelli, Quitalo, Patrizi, and Torres {\cite{CL3}} showed that the  limiting problem becomes a two-phase free boundary problem with the associated free boundary  condition 
$\frac {\partial u_1}{\partial \nu_1} = \frac {\partial u_2}{\partial \nu_2}$,  where $\nu_1$ and $\nu_2$ denote the interior normal directions to the respective supports. This formulation allowed for the application of  the sup-convolutions method (see \cite{Geometric_Free_Boundary} and the references therein), 
to prove that the set of regular points form an open subset of the free boundary and it is locally of class $C^{1,\alpha}$. 

The interaction between the populations in \eqref{adjacentsegregation} is local, meaning that it depends only on the value of the densities, $u_i(x)$, at the point $x$. The segregation is adjacent since the supports of the populations have a common boundary. However, there are many processes where the growth of species $i$ is inhibited by populations $j$ occupying an entire neighborhood around $x$, see for example \cite{CuMaMa, MiEiFang}. Caffarelli, Patrizi, and Quitalo \cite{CL2} introduced system \eqref{main_problems} with the Laplace operator as an example to model non-local interactions. The same system was studied by Chuah, Patrizi, and Torres in \cite{ChPaTo26} with Laplacian replaced by the negative Pucci operator. 
When $H$ is given by \eqref{H1} with $w^2(y)$ in place of $w(y)$, minimizing solutions and the limiting configurations of \eqref{main_problems} have been studied in \cite{NS1,NS2}.    

 The paper is organized as follows. In section 2 we introduce some preliminaries. In section 3, we expand the analysis of regular and singular points in the two-dimensional case from the results in \cite{CL2}. We show that the density at all free boundary points exists and is proportional to the angle, and classify the regular and singular points in terms of densities. In section 4 we generalize the concepts of angles and asymptotic cones to higher dimensions and give a characterization of regular and singular free boundary points in terms of densities and angles (Theorem \ref{Regular_points_Singular_points_in_Higher_Dimensions}). In section 5 we study the structure of the free boundary in higher dimensions (Theorem \ref{Combined_Results_on_Regular_Set}). The analysis in the case the supports of the 
 populations are convex (Theorem \ref{Convex}) is also studied in Section 5.
{



\usetikzlibrary{calc}
\begin{figure}
\begin{tikzpicture}[scale=0.7]
    

    \filldraw[fill=blue!10, draw=blue!80, thick] 
        plot[domain=-4.6:4.6, samples=100] ({0.4*sin(1.5*\x r) - 0.5}, \x) 
        -- (-3.5, 3.1) arc (140:220:4.8) -- cycle;
    \node[blue!80, font=\large] at (-2.2, 0) {$S_1$};

    \filldraw[fill=red!10, draw=red!80, thick] 
        plot[domain=-4.6:4.6, samples=100] ({0.4*sin(1.5*\x r) + 0.5}, \x) 
        -- (3.5, 3.1) arc (40:-40:4.8) -- cycle;
    \node[red!80, font=\large] at (2.2, 0) {$S_2$};

    
\end{tikzpicture}
\caption{Possible configuration with two populations. The supports of the populations are the sets $S_i = \{u_i > 0\} \cap \Omega$. The distance between the the supports is $\sizeofboundary$.}
\label{generalpicture}
\end{figure}
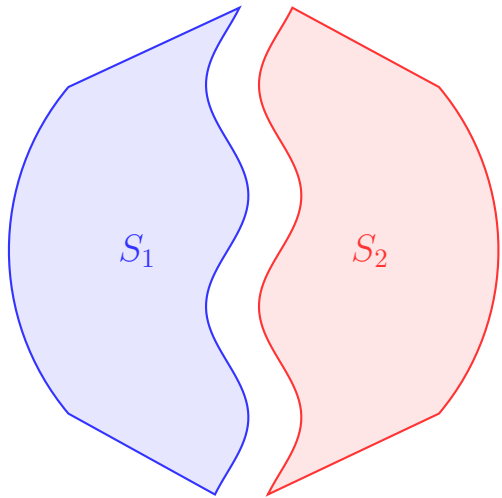

\section{Preliminaries}
In this section we introduce the notation used in this paper and we recall 
some known results. A function $f \in L^1(\mathbb{R}^n)$ is of bounded variation if $|\grad f|$ is a finite Radon measure in $\R^n$.
If the characteristic function of a set $E \subset \R^n$, denoted as $\chi_E$, is of bounded variation, then $E$ is called a set of finite perimeter. If $E \subset \mathbb{R}^n$ is a set of finite perimeter we define 
\begin{equation}
E^{(t)} := \left\{x \in \mathbb{R}^n: \lim_{r \rightarrow 0^+} \frac{|B_r(x) \cap E| }{|B_r(x)|} \textnormal{exists and is equal to } t\right\}
\end{equation}
The essential boundary of $E$ (or measure theoretic boundary) is 
\begin{equation}
\partial^e E := \mathbb{R}^n-(E^{(0)} \cup E^{(1)}),
\end{equation}
and the reduced boundary of $E$ is defined as
\begin{equation}
 \partial ^* E := \left\{x: \nu_E(x) = \lim_{r \rightarrow 0^+}\frac{\grad \chi_E (B_r(x))}{|\grad \chi_E| (B_r(x))} \textnormal{ exists and has length }1\right\}.
\end{equation} 
  
The total variation of the vector-valued Radon measure $\nabla \chi_E$ is the Radon measure $|\nabla \chi_E| = \mathcal{H}^{n-1} \rtangle \partial^*E$.
Also, for any $x, y \in \mathbb{R}^n,$ we denote by $\overline{xy} := \{tx+(1-t)y: t \in [0,1]\}$ the line segment in $\mathbb{R}^n$ with endpoints $x$ and $y.$ For any $x, y, z \in \mathbb{R}^n$ with $x \neq y$ and $z \neq y,$ let $\textnormal{angle}(x;y;z)$ denote the angle between the vectors $x-y$ and $z-y$. The open ball in $\R^n$ centered at $x_0$ and radius $r$ is denoted as $B_r(x_0)$,
while $B^{'}_r(x_0) := \{x \in \mathbb{R}^{n-1} : |x-x_0| < r\}$ is the open ball in $\mathbb{R}^{n-1}$ with radius $r$ centered at $x_0$. 

In this paper, for a convex set $C \subset \mathbb{R}^n$, a boundary point $x_0 \in \partial C$ with a {\it unique} supporting hyperplane will be called a differentiable boundary point. Also, recall that a set is said to be $F_{\sigma}$ if it is a countable union of closed sets, and a set is said to be $G_{\delta}$ if it is a countable intersection of open sets.

We recall the definition of regular and singular points introduced in \cite{CL2}}. In the paper, for any $i \in \{1,...,K\},$ we define by $C_i := \cup_{j \neq i} (\{u_j >0\}\cap \Omega).$

\begin{definition}
A free boundary point $x_0 \in \partial S_i\cap \Omega$ is said to be a {\bf regular point} if there exists a unique $x_1 \in \partial C_i \cap \Omega$ for which $d(x_0,x_1) = \sizeofboundary.$ A free boundary point is said to be {\bf singular} if it is not regular. Hence, $x_0 \in \partial S_i \cap \Omega$ is singular if and only if there are two distinct points $x_1, x_2 \in \partial C_i \cap \Omega, x_1 \neq x_2$ such that $d(x_0,x_1) = d(x_0,x_2) = \sizeofboundary.$
\end{definition}

\begin{definition}
For each population $u_i$ we define
\begin{equation}
S_i:= \{x \in \Omega : u_i (x) > 0 \}
\end{equation}
and therefore $S_i$ is the interior of the support of population $u_i$ in $\Omega$. It was proved in (\cite[Theorem 6.5]{CL2}) that $S_i$ is a set of finite perimeter, see also (\cite[Theorem 4]{ChPaTo26}). 

$\Gamma^{\textnormal{reg}}_i$ is the set of all regular free boundary points in $\partial S_i \cap \Omega$, and $\Gamma^{\textnormal{sin}}_i$ is the set of all singular free boundary points in $\partial S_i \cap \Omega$.

Moreover, in this paper we consider two types of singular points, namely singular points of type 1 and 2 (see Definition \ref{twotypes}). We define $\Gamma^{\textnormal{sin}}_{i,1}$ to be the set of all type 1 singular points in $\partial S_i \cap \Omega$. Also, we define $\Gamma^{\textnormal{sin}}_{i,2}$ to be the set of all type 2 singular points in $\partial S_i \cap \Omega$.

 A singular free boundary point $x_0 \in \partial S_i \cap \Omega$ is said to be a "cusp" if it is a point of density zero for the set $\{ u_i > 0 \}$. We define $\Gamma^{\textnormal{cusp}}_i$ to be the set of all cusps in $\partial S_i \cap \Omega$.
\end{definition}
We recall the following two classical theorems in geometric measure theory:
\begin{theorem}
 \label{Federer_Theorem}
Let $E \subset \mathbb{R}^n$ be a set of locally finite perimeter, then we have
\begin{eqnarray}
    \partial^* E \subset E^{(\frac{1}{2})} \subset \partial ^e E, \nonumber
\end{eqnarray}
with $\mathcal{H}^{n-1}(\partial^e E-\partial^* E) = 0.$
\end{theorem}
{

\begin{theorem}
\label{structure_of_the_reduced_boundary}
(\cite{Enrico}, Theorem 4.4)
If $E \subset \mathbb{R}^n$ is a set of finite perimeter, then we have
\begin{eqnarray}
\partial ^*E = \left(\bigcup_{i \in \mathbb{N}} C_i\right) \cup N, \nonumber
\end{eqnarray}    
where each $C_i$ is compact and $C_i \subset f^{-1}(0)$, where $f$ is a $C^1$ function defined in a neighborhood of $C_i$. Moreover, $\mathcal{H}^{n-1}(N) = 0$. 
\end{theorem}
}
We will use the following closed graph theorem in the analysis of the regular set:
\begin{lemma}
\label{The_Closed_Graph_Theorem_in_Point_Set_Topology}
Assume that $X$ is a Hausdorff topological space, and suppose that $Y$ is a compact Hausdorff topological space. If $f: X \rightarrow Y$ is such that the graph $\{(x,f(x)): x \in X\}$ is closed in $X \times Y$ (under the product topology), then $f$ is continuous on $X.$
\end{lemma}

\section{A Characterization of Regular and Singular Points in two dimensions}

The main result of this section is Theorem \ref{A_Dichotomy_for_Free_Boundary_Point} that characterizes singular and regular points in terms of densities in the two dimensional case. We start this section by recalling the characterization of regular and singular points in terms of the angle of the asymptotic cone given in \cite{CL2} (for $n=2$):

\begin{lemma}
\label{fortwo}
Consider $S_i := \{u_i > 0\} \cap \Omega$ and $C_i := \cup_{j \neq i} S_j$ for all $i = 1, ..., K.$ 
Let $x_0 \in \partial S_i \cap \Omega$ be a free boundary point. 
 The {\bf convex asymptotic cone} centered at $x_0$ is defined as the intersection of closed half spaces as follows:
\begin{equation}
\bigcap_{y \in \partial C_i \cap \Omega, d(x_0,y) = \sizeofboundary} \{p \in \mathbb{R}^2:(p-x_0)\cdot (y-x_0) \leq 0)\} 
\end{equation}
The angle at $x_0 \in \partial S_i \cap \Omega,$ denoted by $\textnormal{angle}(x_0),$ is defined to be the opening of the asymptotic cone. Then, a free boundary point is regular if and only if $\textnormal{angle}(x_0)=\pi$. 
\end{lemma}



In this section, we provide another characterization of regular and singular points in terms of densities. We assume that the function $H$ is given by \eqref{H1} or \eqref{H2}.  


The following two results were proven in \cite{CL2}.
\begin{lemma}
\label{basic_geometric_lemma}
(\cite{CL2} Lemma 8.1)
Assume that the origin $(0,0) \in \partial S_i \cap \Omega$ is a free boundary point, and let $C := \{(x_1,x_2) \in \mathbb{R}^2: x_2 \geq \alpha|x_1|\},$ where $\alpha > 0,$ be the asymptotic cone of $S_i$ at $(0,0) \in \partial S_i \cap \Omega.$ Then there exist $y_1, y_2 \in \partial C_i \cap \Omega$ such that the ball $B_{\sizeofboundary}(y_1)$ is tangent to the line $x_2 = \alpha x_1$ at $(0,0)$, and the ball $B_{\sizeofboundary}(y_2)$ is tangent to the line $x_2 = -\alpha x_1$ at $(0,0)$.
\end{lemma}

\begin{lemma}
\label{lemma_on_Lipschitz_graph}
(\cite{CL2} Lemma 8.2 and Lemma 8.3)
Assume $x_0 \in \partial S_i \cap \Omega$ is such that $\textnormal{angle}(x_0) \in (0,\pi]$. Then there exists a neighborhood $U$ of $x_0,$ such that upon rotating and translating the coordinate axes $(x_1,x_2)$ if necessary, there is a locally Lipschitz function $\psi: [-r,r] \rightarrow \mathbb{R}$, for some small $r > 0,$ such that $\partial S_i \cap U = \{(x_1,\psi(x_1)): x_1 \in [-r,r]\}.$ If in addition the free boundary point is regular (i.e., $\textnormal{angle}(x_0) = \pi$,) the function $\psi$ is differentiable at the point. Moreover, if there is an open set $U \subset \Omega$ such that every free boundary point of $\partial S_i \cap U$ is regular, we have that $\partial S_i \cap U$ is locally a $C^1$ curve in the plane.
\end{lemma}

We will need the following Lemma, which says that the angle at a free boundary point reflects the local geometry near the point.

\begin{lemma}
\label{key_lemma_for_the_limit_density}
Suppose that $x_0 \in \partial S_i \cap \Omega$ is a free boundary point with $\textnormal{angle}(x_0) \in (0,\pi).$ Then:
\begin{enumerate}
\item  For any small $\epsilon > 0$, there exists a $\delta > 0$ and a cone $\mathcal{K}_1$ with vertex $x_0$ and opening $\textnormal{angle}(x_0)-\epsilon$ such that $\mathcal{K}_1 \cap B_{r}(x_0) \subset S_i \cap B_{r}(x_0)$ for any $r < \delta.$
\item For any small $\epsilon > 0$, there exists a $\delta > 0$ and a cone $\mathcal{K}_2$ with vertex $x_0$ and opening $\textnormal{angle}(x_0)+\epsilon$ such that $S_i \cap B_{r}(x_0)  \subset \mathcal{K}_2 \cap B_{r}(x_0)$ for any $r < \delta.$
\end{enumerate}
\end{lemma} 
\begin{proof}   

{\bf 1.} 
Suppose that the asymptotic cone is $\{(x,y) \in \mathbb{R}^2:  y \geq \alpha|x|\}.$ Take $y_1 := (\sizeofboundary \frac{\alpha}{\sqrt{1+\alpha^2}},-\sizeofboundary \frac{1}{\sqrt{1+\alpha^2}})$ and $y_2 := (-\sizeofboundary \frac{\alpha}{\sqrt{1+\alpha^2}},-\sizeofboundary \frac{1}{\sqrt{1+\alpha^2}})$ ($y_1$ and $y_2$ are the red points in Figure \ref{2D_circle}). Then it follows that $\tan (\frac{\textnormal{angle}(x_0)}{2}) = \frac{1}{\alpha},$ and that for any $z = (z_1, z_2) \in \partial C_i \cap \Omega$ satisfying $d(z,x_0) = d(z_0,(0,0)) = \sizeofboundary,$ we have $z_2 \leq -\sizeofboundary \frac{1}{\sqrt{1+\alpha^2}}$. 
Let $\epsilon > 0$ and take $\gamma > \alpha$ with the property that $\tan (\frac{\textnormal{angle}(x_0)-\epsilon}{2}) = \frac{1}{\gamma}$. Let $\mathcal{K}_1$ be the cone with vertex at the origin, opening $\textnormal{angle}(x_0)-\epsilon$ and with central axis at the positive $y$- axis. 
Take also $q_1 := (\sizeofboundary \frac{\gamma}{\sqrt{1+\gamma^2}},-\sizeofboundary \frac{1}{\sqrt{1+\gamma^2}})$ and $q_2 := (-\sizeofboundary \frac{\gamma}{\sqrt{1+\gamma^2}},-\sizeofboundary \frac{1}{\sqrt{1+\gamma^2}}),$ and let $Q \subset \partial B_{\sizeofboundary}(x_0)$ be the shorter arc on the circle $\partial B_{\sizeofboundary}(x_0)$ joining $q_1$ and $q_2$ ($q_1$ and $q_2$ are the blue points in Figure \ref{2D_circle}).
By Lemma \ref{lemma_on_Lipschitz_graph}, $\partial S_i \cap \Omega$ is locally the graph of a Lipschitz continuous function in a neighborhood of the origin. 


 We now show that there exists a $\delta > 0$ such that $\psi(x) \leq \gamma x$ for any $|x| < \delta.$ To show this, assume by contradiction that there exists a sequence $\{r_k\}_{k \in \mathbb{N}}$ with $r_1 > r_2 > \cdots,$ $r_k \rightarrow 0$ as $k \rightarrow \infty,$ and $\psi(r_k) > \gamma r_k$ for all $k \in \mathbb{N}.$ For each $k \in \mathbb{N},$ we pick a $t_k \in \partial C_i \cap \Omega$ with $d(t_k,(r_k,\psi(r_k))) = \sizeofboundary.$ 
Then $B_{\sizeofboundary}(t_k)$ is an exterior tangent ball to $S_i$ at $(r_k,\psi(r_k)).$
Now consider the triangle with vertices $x_0, t_k,$ and $(r_k,\psi(r_k)).$ Observe that $d((r_k,\psi(r_k)),t_k) = \sizeofboundary$ and $d(x_0,t_k) \geq \sizeofboundary,$ so $\textnormal{angle}(t_k;x_0;(r_k,\psi(r_k))) \leq \textnormal{angle}(x_0;(r_k,\psi(r_k));t_k),$ which implies $\textnormal{angle}(t_k;x_0;(r_k,\psi(r_k))) \leq \frac{\pi}{2}.$ Now consider the normalization $\frac{\sizeofboundary t_k}{|t_k|} \in \partial B_{\sizeofboundary}(x_0).$ Since 
\begin{equation*}
    \textnormal{angle}\left(t_k;x_0;(r_k,\psi(r_k))\right) = \textnormal{angle}\left(\frac{\sizeofboundary t_k}{|t_k|};x_0;(r_k,\psi(r_k))\right) \leq \frac{\pi}{2} \nonumber
\end{equation*}
and $\psi(r_k) \geq \gamma r_k$ for all $k \in \mathbb{N},$ we have that $\frac{\sizeofboundary t_k}{|t_k|} \in \partial B_{\sizeofboundary}(x_0)\setminus Q$ for all $k.$ Also note that since each $B_{\sizeofboundary}(t_k)$ is an exterior tangent ball to $S_i$ at $(r_k,\psi(r_k)),$ $\{t_k\}_{k \in \mathbb{N}}$ converges along a subsequence to some $s$ such that $B_{\sizeofboundary}(s)$ is an exterior tangent ball at $x_0.$ As $|t_k| \rightarrow \sizeofboundary$ along this subsequence, we have that $\frac{\sizeofboundary t_k}{|t_k|} \rightarrow s$ along this subsequence. We get that $s \in \overline{\partial B_{\sizeofboundary}(x_0)\setminus Q},$ so $s$ is not in the shorter arc on $\partial B_{\sizeofboundary}(x_0)$ joining $y_1$ and $y_2.$ We get a contradiction. We have shown that there exists a $\delta > 0$ such that $\psi(x) \leq \gamma x$ for any $|x| < \delta.$ It follows that $\mathcal{K}_1 \cap B_{r}(x_0) \subset S_i \cap B_{r}(x_0)$ for any $r$ sufficiently small.

\begin{figure}
\begin{tikzpicture}
\draw (0,0) circle[radius = 2.5 cm];
\draw [blue](0,0)--(1.5,2);
\draw [red] (0,0)--(2,1.5);
\draw (0,0)--(-1.5,2);
\draw (0,0)--(-2,1.5);
\draw (0,0)--(2,-1.5);
\draw (0,0)--(1.5,-2);
\draw (0,0)--(-2,-1.5);
\draw (0,0)--(-1.5,-2);
\filldraw[thick, draw = red, fill = red!60!red!60] (1.5,-2) circle (0.1);
\filldraw[thick, draw = red, fill = red!60!red!60] (-1.5,-2) circle (0.1);
\filldraw[thick, draw = blue, fill = blue!60!blue!60] (2,-1.5) circle (0.1);
\filldraw[thick, draw = blue, fill = blue!60!blue!60] (-2,-1.5) circle (0.1);
\filldraw[thick, draw = green, fill = green!60!green!60] (0,0) circle (0.1);
\filldraw[thick, draw = gray, fill = gray!60!gray!60] (0.1,0.4) circle (0.1);
\filldraw[thick, draw = gray, fill = gray!60!gray!60] (0.2,0.6) circle (0.1);
\filldraw[thick, draw = gray, fill = gray!60!gray!60] (0.4,1.2) circle (0.1);
\filldraw[thick, draw = gray, fill = gray!60!gray!60] (0.3,0.8) circle (0.1);
\filldraw[thick, draw = yellow, fill = yellow!60!yellow!60] (-32:2.5) circle (0.1);
\filldraw[thick, draw = yellow, fill = yellow!60!yellow!60] (-31:2.5) circle (0.1);
\filldraw[thick, draw = yellow, fill = yellow!60!yellow!60] (-29:2.5) circle (0.1);
\filldraw[thick, draw = yellow, fill = yellow!60!yellow!60] (-27:2.5) circle (0.1);
\end{tikzpicture}
\caption{This figure illustrates Lemma \ref{key_lemma_for_the_limit_density}: $L_1$ is the red line $\{(x,y): y = \alpha x, x \geq 0\};$ $L_2$ is the blue line $\{(x,y): y = \gamma x, x \geq 0\}.$ The red points are $y_1 = (\sizeofboundary \frac{\alpha}{\sqrt{1+\alpha^2}},-\sizeofboundary \frac{1}{\sqrt{1+\alpha^2}})$ and $y_2 = (-\sizeofboundary \frac{\alpha}{\sqrt{1+\alpha^2}},-\sizeofboundary \frac{1}{\sqrt{1+\alpha^2}}).$ The blue points are $q_1 = (\sizeofboundary \frac{\gamma}{\sqrt{1+\gamma^2}},-\sizeofboundary \frac{1}{\sqrt{1+\gamma^2}})$ and $q_2 = (-\sizeofboundary \frac{\gamma}{\sqrt{1+\gamma^2}},-\sizeofboundary \frac{1}{\sqrt{1+\gamma^2}}).$ The yellow points on the right of $q_1$ are the points $\frac{\sizeofboundary t_k}{|t_k|}.$ The gray points are $(r_k,\psi(r_k)).$}
\label{2D_circle}
\end{figure}
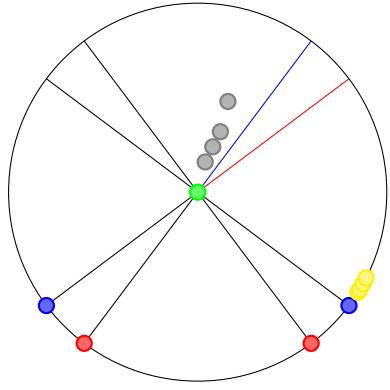

{\bf 2.} 
Suppose that the cone is $\{(x,y) \in \mathbb{R}^2:  y \geq \alpha|x|\}$. Take $y_1 := (\sizeofboundary \frac{\alpha}{\sqrt{1+\alpha^2}},-\sizeofboundary \frac{1}{\sqrt{1+\alpha^2}})$ and $y_2 := (-\sizeofboundary \frac{\alpha}{\sqrt{1+\alpha^2}},-\sizeofboundary \frac{1}{\sqrt{1+\alpha^2}})$. 
Then it follows that $\tan (\frac{\textnormal{angle}(x_0)}{2}) = \frac{1}{\alpha}$ and that for any $z = (z_1, z_2) \in \partial C_i \cap \Omega$ satisfying $d(z,x_0) = d(z,(0,0)) = \sizeofboundary,$ we have $z_2 \leq -\sizeofboundary \frac{1}{\sqrt{1+\alpha^2}}$ 
Now let $\epsilon > 0$ and take  $\beta \in (0,\alpha)$ with the property that $\tan (\frac{\textnormal{angle}(x_0)+\epsilon}{2}) = \frac{1}{\beta}.$ Assume that the line $\{(x,y) \in \mathbb{R}^2: y = \beta x\}$ intersects the circle $\partial B_{\sizeofboundary}(y_1)$ at $x_0$ and $w_1$, while the line $\{(x,y) \in \mathbb{R}^2: y = -\beta x\}$ intersects the circle $\partial B_{\sizeofboundary}(y_2)$ at the $x_0$ and $w_2$. Take $\delta := d((0,0),w_1) = d((0,0),w_2),$ and let $\mathcal{K}_2$ be the cone with vertex at the origin, opening $\textnormal{angle}(x_0)+\epsilon,$ and with central axis at the positive $y$-axis. It is clear that $S_i \cap B_{r}(x_0)  \subset \mathcal{K}_2 \cap B_{r}(x_0)$ for any $r < \delta.$
\end{proof}

{In Lemma \ref{key_lemma_for_the_limit_density}, {\bf 1.} also holds for $\textnormal{angle}(x_0) = \pi$, and {\bf 2.} also holds for $\textnormal{angle}(x_0) = 0$.}
We next show, using Lemma \ref{key_lemma_for_the_limit_density}, that the density exists at all the free boundary points, and use this fact to obtain a structure theorem for the free boundary set.                 
\begin{theorem}
\label{A_Dichotomy_for_Free_Boundary_Point}
For $n = 2,$ let $H$ be given by \eqref{H1} or \eqref{H2}, and let $x_0 \in \partial S_i \cap \Omega$ be a free boundary point. Then $\lim_{r \rightarrow 0^+} \frac{|S_i \cap B_r(x_0)|}{|B_r(x_0)|}$ exists and is equal to $\frac{\textnormal{angle}(x_0)}{2\pi},$ where $\textnormal{angle}(x_0) \in [0,\pi]$.

Moreover, the following three conditions (a)(b)(c) are equivalent:
\\(a) $x_0$ is a regular free boundary point.
\\(b) $\textnormal{angle}(x_0) = \pi.$
\\(c) $\lim_{r \rightarrow 0^+} \frac{|S_i \cap B_r(x_0)|}{|B_r(x_0)|} = \frac{1}{2}.$
\\Also, the following three conditions (d)(e)(f) are equivalent:
\\(d) $x_0$ is a singular free boundary point.
\\(e) $\textnormal{angle}(x_0) < \pi.$
\\(f) $\lim_{r \rightarrow 0^+} \frac{|S_i \cap B_r(x_0)|}{|B_r(x_0)|} < \frac{1}{2}.$
\end{theorem}   
\begin{proof}          
We first show that $\lim_{r \rightarrow 0^+} \frac{|S_i \cap B_r(x_0)|}{|B_r(x_0)|}$ exists and is equal to $\frac{\textnormal{angle}(x_0)}{2\pi},$ where $\textnormal{angle}(x_0) \in [0,\pi]$. 
To see this, consider first the case where $\textnormal{angle}(x_0) \in (0,\pi)$. Let $\epsilon > 0$. By Lemma \ref{key_lemma_for_the_limit_density} (1), there exists a $\delta > 0$ and a cone $\mathcal{K}_1$ with vertex $x_0$ and opening $\textnormal{angle}(x_0)-\epsilon$ such that $\mathcal{K}_1 \cap B_{r}(x_0) \subset S_i \cap B_{r}(x_0)$ for any $r < \delta.$ Therefore, for any $r \in (0,\delta),$ we have 
\begin{eqnarray}
\label{lower_bound_for_the_density}
    \frac{\textnormal{angle}(x_0)-\epsilon}{2\pi} = \frac{|\mathcal{K}_1 \cap B_{r}(x_0)|}{|B_r(x_0)|} \leq \frac{|S_i \cap B_{r}(x_0)|}{|B_r(x_0)|}. 
\end{eqnarray}
Hence, since $\epsilon$ is arbitrary,
\begin{eqnarray}
\label{lower_bound_for_the_density_limit_inf_noeps}
\frac{\textnormal{angle}(x_0)}{2\pi} \leq \liminf_{r \rightarrow 0^+}\frac{|S_i \cap B_{r}(x_0)|}{|B_r(x_0)|}. 
\end{eqnarray}
Similarly,  
\begin{eqnarray}
\label{upper_bound_for_the_density_limit_sup_noeps}
    \frac{\textnormal{angle}(x_0)}{2\pi} \geq \limsup_{r \rightarrow 0^+}\frac{|S_i \cap B_{r}(x_0)|}{|B_r(x_0)|}. 
\end{eqnarray}
Combining \eqref{lower_bound_for_the_density_limit_inf_noeps} and \eqref{upper_bound_for_the_density_limit_sup_noeps}, it follows that $\lim_{r \rightarrow 0^+} \frac{|S_i \cap B_r(x_0)|}{|B_r(x_0)|}$ exists and is equal to $\frac{\textnormal{angle}(x_0)}{2\pi}$. The case where $\textnormal{angle}(x_0) = 0$ or $\pi$ can be dealt with similarly, since Lemma \ref{key_lemma_for_the_limit_density} {\bf 1.} also holds for $\textnormal{angle}(x_0) = \pi$, and Lemma \ref{key_lemma_for_the_limit_density} {\bf 2.} also holds for $\textnormal{angle}(x_0) = 0$. Note that $\textnormal{angle}(x_0) \in [0,\pi]$ implies that $\lim_{r \rightarrow 0^+} \frac{|S_i \cap B_r(x_0)|}{|B_r(x_0)|} \leq \frac{1}{2}.$

The equivalence of (a) and (b) follows from immediately from the definition; the equivalence of (b) and (c) follows from the equality $\lim_{r \rightarrow 0^+} \frac{|S_i \cap B_r(x_0)|}{|B_r(x_0)|} = \frac{\textnormal{angle}(x_0)}{2\pi}.$ The conditions (d)(e)(f) are equivalent since (a)(b)(c) are equivalent. 

\end{proof}

By Theorem \ref{A_Dichotomy_for_Free_Boundary_Point}, we have the following:
\begin{remark}
\label{one_criterion_for_regular_points}
Let $x_0 \in \partial S_i \cap \Omega$. If $\partial S_i \cap \Omega$ is a straight line in a neighborhood of $x_0$ then $x_0$ is a regular point. That is, there exists a unique $y \in \partial C_i \cap \Omega$ such that $d(x_0,y) = \sizeofboundary$.
\end{remark}

\section{Density and Angle of a Free Boundary Point in Higher Dimensions}
In this section, we generalize the concepts of angles from the two-dimensional case to higher dimensions and study their relationships with densities. The main theorem of this section is Theorem \ref{Regular_points_Singular_points_in_Higher_Dimensions}. We recall that $C_i = \cup_{j \neq i}S_j$ for all $i = 1, \cdots, K$.

\begin{definition}
\label{Angle_in_Higher_Dimension}
    Let $x_0 \in \partial \{u_i>0\} \cap \Omega$ be a free boundary point.

{\bf (a)} For each $y \in \partial C_i \cap \Omega$ with $d(x_0,y) = \sizeofboundary$, we define the closed semi-sphere
\begin{eqnarray}
    S^+(x_0;y) := \{x \in \partial B_{\sizeofboundary}(x_0): (x-x_0)\cdot(y-x_0) \leq 0.\}
\end{eqnarray}
That is, $S^+(x_0;y)$ is the closed semi-sphere with radius $\sizeofboundary$ centered at $x_0,$ consisting of all points $x \in \partial B_{\sizeofboundary}(x_0)$ such that the angle between the vectors $x-x_0$ and $y-x_0$ is at least $\frac{\pi}{2}$.

{\bf (b)} The {\bf angle} at $x_0,$ denoted by $\textnormal{angle}(x_0),$ is defined by
\begin{eqnarray}
\label{equation_for_computing_angles}
    \textnormal{angle}(x_0) := \frac{\mathcal{H}^{n-1}(\bigcap_{y \in \partial C_i \cap \Omega, d(x_0,y) = \sizeofboundary} S^+(x_0;y))}{\sizeofboundary^{n-1}}.
\end{eqnarray}
The {\bf asymptotic convex cone} at $x_0,$ denoted by $\textnormal{conv}(x_0),$ is defined by 
\begin{equation}
\textnormal{conv}(x_0) := \bigcap_{y \in \partial C_i \cap \Omega, d(x_0,y) = \sizeofboundary}\{p \in \mathbb{R}^n: (p-x_0)\cdot (y-x_0) \leq 0.\}
\end{equation}
Note that $\textnormal{angle}(x_0) = \frac{\mathcal{H}^{n-1}(\partial B_{\sizeofboundary}(x_0) \cap \textnormal{conv}(x_0))}{\sizeofboundary^{n-1}}.$
 Also, note that if $n = 2$, $\textnormal{angle}(x_0)$ coincides with the one given in Lemma \ref{fortwo}.
\end{definition} 
\begin{remark}
For any free boundary point $x_0 \in \partial S_i \cap \Omega$ we have $\textnormal{angle}(x_0) \in [0,\frac{n\omega_n}{2}].$ When $n = 2$, $\omega_2 = \pi,$ so $\textnormal{angle}(x_0) \in [0,\pi],$ which is consistent with Lemma \ref{fortwo}.
$\Box$
\end{remark}

We proceed to show that near a regular free boundary point, $\partial S_i \cap \Omega$ is locally the graph of a differentiable function. The proof extends Lemma 8.2 in \cite{CL2} to arbitrary dimensions.

\begin{lemma}
\label{higher_dimensional_free_boundary_Lipschitz_graph} 
Let $x_0 \in \partial S_i \cap \Omega$ be a regular free boundary point. Then there exists a neighborhood $U$ of $x_0,$ such that upon rotating and translating the coordinate axes $(x_1,x_2,...,x_n)$ if necessary, there is a continuous function $\psi: W \rightarrow \mathbb{R}$, where $W$ is a neighborhood containing the origin in $\mathbb{R}^{n-1}$ such that, in the coordinates $(x_1, x_2, ..., x_n),$ we have
\begin{eqnarray}
    \partial S_i \cap U = \{(x,\psi(x)):x \in W\}. \nonumber
\end{eqnarray}
Moreover, $\psi(0,\cdots,0) = 0,$ and $\psi$ is differentiable at $(0,\cdots,0)$ with $D\psi(0,\cdots,0) = (0,\cdots,0)$.   
\end{lemma}        
\begin{proof} 
Since $x_0 \in \partial S_i \cap \Omega$ is a regular point, there exists a unique $y_0 \in \partial C_i \cap \Omega$ for which $d(x_0,y_0) = \sizeofboundary.$ We fix a system of coordinates $(x_1,x_2, ..., x_n)$ such that the $x_n $ axis coincides with the line segment $\overline {x_0 y_0}.$ Also, the $x_n $  axis is oriented in such a way that positive $x_n $ axis points in the direction of the vector from $y_0$ to $x_0.$ Choose $x_1, x_2, ..., x_{n-1}$ so that $(x_1, x_2, ..., x_n)$ is orthonormal. Then it follows that under these coordinates, $x_0 = (0,0...,0),$ $y_0 = (0,0,...,0,-\sizeofboundary)$ and the convex asymptotic cone $\mathcal{C}$ of $S_i$ at $x_0$ is given by $\mathcal{C} = \{x_n \geq 0\}.$ The proof is divided in three steps:          

{\bf Step 1:} We first show that $\partial S_i \cap \Omega$ is the graph of a function $\psi$ in a small neighborhood $W$ of the origin in $\mathbb{R}^{n-1}.$ It suffices to prove that there is an $r > 0$ small enough such that for any $t \in \mathbb{R}^{n-1}$ with $|t| < r,$ the vertical line $\{(x_1,x_2, ..., x_{n-1}) = t\}$ intersects $\partial S_i \cap B_r(0,\cdots,0)$ in exactly one point. Assume by contradiction that there exists a sequence $\{t_k\}_{k \in \mathbb{N}} \subset \mathbb{R}^{n-1},$ with $|t_k| \rightarrow 0$ as $k \rightarrow \infty,$ such that the line $\{(x_1,x_2, ..., x_{n-1}) = t_k\}$ intersects $\partial S_i \cap B_r(0)$ at two distinct points $(t_k,a_k)$ and $(t_k,b_k)$ with $b_k > a_k.$
Observe that $S_i \cap B_{\sizeofboundary}(y_0) = \emptyset,$ so for all $k \in \mathbb{N}$ sufficiently large, both $(t_k,a_k)$ and $(t_k,b_k)$ are above the ball $B_{\sizeofboundary}(y_0).$ 
Next, for all $k \in \mathbb{N},$ let $y^a_k$ and $y^b_k$ be points in $\partial C_i \cap \Omega$ where $(t_k,a_k)$ and $(t_k,b_k)$ realizes the distances $\sizeofboundary$ from $\partial C_i \cap \Omega$ (the existence of $y^a_k \in \partial C_i \cap \Omega$ and $y^b_k \in \partial C_i \cap \Omega$ follows from Theorem 7.1 in \cite{CL2}). Then the balls $B_{\sizeofboundary}(y^a_k)$ and $B_{\sizeofboundary}(y^b_k)$ are exterior tangent balls to $\partial S_i$ at $(t_k,a_k)$ and $(t_k,b_k),$ respectively. 
Note that $y^a_k$ must be in the lower half ball $\partial B_{\sizeofboundary}(t_k,a_k) \cap \{x_n \leq a_k\}$ for $k$ sufficiently large. (for otherwise, since $d((t_k,a_k),(t_k,b_k)) \rightarrow 0$ as $k \rightarrow \infty,$ the tangent ball $B_{\sizeofboundary}(y^a_k)$ would contain $(t_k,b_k)$ for $k$ large enough, contradicting $d((t_k,b_k), \partial C_i) \geq \sizeofboundary$). Similarly, $y^b_k$ has to belong to the upper half ball $\partial B_{\sizeofboundary}(t_k,b_k) \cap \{x_n \geq b_k\}$ for all $k$ sufficiently large. This implies that the tangent balls $B_{\sizeofboundary}(y^b_k)$ converges (along a subsequence) to a tangent ball to $S_i$ at $(0,0...,0),$ say $B_{\sizeofboundary}(y^b),$ with $y^b \in \{x_n \geq 0\}.$ However, we already have that $B_{\sizeofboundary}(y_0)$ is the unique tangent ball at $x_0.$ Consequently, we have $y_0 = y^b \in \{x_n \geq 0\}.$ This contradicts to the fact that $y_0 = (0,0,...,0,-\sizeofboundary).$ Hence, $\partial S_i \cap \Omega$ is the graph of a function $\psi$ in a neighborhood  of $x_0$ in $\mathbb{R}^{n-1}$, and Step 1 is proved.

{\bf Step 2}: We show that the function $\psi$ is continuous near the origin. We note that the graph of $\psi$ over $\overline {B_{\frac{r}{2}}(0,\cdots,0)}$ is closed in $\mathbb{R}^n$, so from Lemma \ref{The_Closed_Graph_Theorem_in_Point_Set_Topology} it follows that the function $\psi$ is continuous near the origin.      

{\bf Step 3}: In order to establish that $\psi$ is differentiable at the origin with $D\psi(0,\cdots,0) = (0,\cdots,0)$
we show that $\lim_{h \rightarrow 0}\frac{|\psi(h)|}{|h|} = 0$. Assume by contradiction that this fails, so there is a small $\delta > 0$ and a sequence $\{h_k\}_{k \in \mathbb{N}}$ with $|h_k| \downarrow 0$ and $\frac{|\psi(h_k)|}{|h_k|} \geq \delta$ for all $k \in \mathbb{N}$, so $|\psi(h_k)| \geq \delta |h_k|$ for all $k \in \mathbb{N}$. Since the graph of $\psi$ is above the ball centered at $(0,\cdots,-\sizeofboundary)$ with radius $\sizeofboundary$, we see that in fact $\psi(h_k) \geq \delta |h_k|$ for all $k \in \mathbb{N}$ (i.e., each point $(h_k,\psi(h_k))$ is in the cone $K_{\delta} := \{(x,y) \in \mathbb{R}^{n-1} \times \mathbb{R}: y \geq \delta |x|\}$). For each $k \in \mathbb{N}$ we pick a point $p_k \in \partial C_i \cap \Omega$ with $d(p_k,(h_k,\psi(h_k))) = \sizeofboundary$, and so $B_{\sizeofboundary}(p_k)$ is a exterior tangent ball of $S_i \cap \Omega$ at $(h_k,\psi(h_k))$. Let $p \in \partial C_i \cap \Omega$ be a subsequential limit of $\{p_k\}_{k \in \mathbb{N}}$, so $B_{\sizeofboundary}(p)$ is an exterior tangent ball at $(0,\cdots,0)$. Since each point $(h_k,\psi(h_k))$ is in the cone $K_{\delta} := \{(x,y) \in \mathbb{R}^{n-1} \times \mathbb{R}: y \geq \delta |x|\}$, we have that $B_{\sizeofboundary}(p) \neq B_{\sizeofboundary}(y_0)$. This contradicts to the fact that $x_0$ is a regular point. Hence $\psi$ is differentiable at the origin with $D\psi(0,\cdots,0) = (0,\cdots,0)$.

\end{proof}
We will need the following lemma, whose proof is standard and is therefore omitted.   
\begin{lemma}
\label{density_lemma}
Let $U \subset \mathbb{R}^{n-1}$ be an open set containing the origin, and let $f : U \rightarrow \mathbb{R}$ be continuous on $U$ and differentiable at the origin with $f(0,\cdots,0) = 0,$ and $Df(0,\cdots,0) = (0,\cdots,0)$. Take $W := \{(x,y): x \in U, y \in \mathbb{R}, y > f(x)\} \subset \mathbb{R}^n$ to be the epigraph of $f$ over $U$. Assume moreover that $W$ satisfies an exterior sphere condition at the origin. 
Then we have
\begin{eqnarray}
\label{geometry_lemma_density_is_one_half}
    \lim_{r \rightarrow 0^+}\frac{|B_r(0,\cdots,0) \cap W|}{|B_r(0,\cdots,0)|} = \frac{1}{2}.
\end{eqnarray}
\end{lemma}

By combining the above results, we are able to prove the following characterization of regular and singular points. In what follows, $\omega_n := \frac{\pi^{\frac{n}{2}}}{\Gamma(\frac{n}{2}+1)}$ is the volume of the unit closed ball in $\mathbb{R}^n$. 

\begin{theorem}
\label{Regular_points_Singular_points_in_Higher_Dimensions}
In arbitrary dimension, let $H$ be given by \eqref{H1} or \eqref{H2}. Then for $\mathcal{H}^{n-1}$-a.e. $x_0 \in \partial \{u_i > 0\} \cap \Omega,$ the density $\lim_{r \rightarrow 0^+} \frac{|B_r(x_0) \cap S_i|}{|B_r(x_0)|}$ exists.
\\ Moreover, for a free boundary point $x_0 \in \partial S_i \cap \Omega$, the following three conditions are  equivalent:
\\(a) $x_0$ is a regular free boundary point.
\\(b) $\textnormal{angle}(x_0) = \frac{n\omega_n}{2},$ (in the sense of Definition \ref{Angle_in_Higher_Dimension}.) 
\\(c) $\lim_{r \rightarrow 0^+}\frac{|B_r(x_0)\cap S_i|}{|B_r(x_0)|} = \frac{1}{2}.$ 
\\In addition, the following conditions are also equivalent:
\\(d) $x_0$ is a singular free boundary point.
\\(e) $\textnormal{angle}(x_0) < \frac{n\omega_n}{2},$ (in the sense of Definition \ref{Angle_in_Higher_Dimension}.) 
\\(f) $\limsup_{r \rightarrow 0^+}\frac{|B_r(x_0)\cap S_i|}{|B_r(x_0)|} < \frac{1}{2}.$
\end{theorem}     
\begin{proof} 
{   

We first check that for a free boundary point $x_0 \in \partial \{u_i > 0\} \cap \Omega$, the conditions (a)(b)(c) are all equivalent. 
   
Assume that (a) holds, so there is a unique $z \in \partial C_i \cap \Omega$ for which $d(x_0,z) = \sizeofboundary.$ Hence, $\cap_{y \in \partial C_i \cap \Omega, d(x_0,y) = \sizeofboundary} S^+(x_0;y) = S^+(x_0;z),$ and $\mathcal{H}^{n-1}(\cap_{y \in \partial C_i \cap \Omega, d(x_0,y) = \sizeofboundary} S^+(x_0;y)) = \mathcal{H}^{n-1}(S^+(x_0;z)) = \frac{1}{2}n\omega_n \sizeofboundary^{n-1}.$ This shows that
\begin{equation}
    \textnormal{angle}(x_0) = \frac{\frac{1}{2}n\omega_n \sizeofboundary^{n-1}}{\sizeofboundary^{n-1}} = \frac{n \omega_n}{2}, \nonumber 
\end{equation}
which is (b). 

We now show that (b) implies (a). Assume that (b) holds, so $\frac{n\omega_n}{2} = \textnormal{angle}(x_0) = \frac{\mathcal{H}^{n-1}(\cap_{y \in \partial C_i \cap \Omega, d(x_0,y) = \sizeofboundary} S^+(x_0;y))}{\sizeofboundary^{n-1}}.$ This shows that 
\begin{eqnarray}
\label{Area_of_The_Intersection_1}
\mathcal{H}^{n-1}(\cap_{y \in \partial C_i \cap \Omega, d(x_0,y) = \sizeofboundary} S^+(x_0;y)) = \frac{n\omega_n \sizeofboundary^{n-1}}{2}. 
\end{eqnarray}
If there exists $z_1, z_2 \in \partial C_i \cap \Omega$ with $z_1 \neq z_2$ and $d(z_1,x_0) = d(z_2,x_0) = \sizeofboundary$, then $S^+(x_0;z_1) \cap S^+(x_0;z_2)$ is properly contained in $S^+(x_0;z_1)$, and hence 
\begin{equation}
\mathcal{H}^{n-1}(S^+(x_0;z_1) \cap S^+(x_0;z_2)) < \frac{n\omega_n \sizeofboundary^{n-1}}{2},
\end{equation}
contradicting \eqref{Area_of_The_Intersection_1}. This shows that there is a unique $z \in \partial C_i \cap \Omega$ satisfying $d(z,x_0) = \sizeofboundary.$ Thus, $x_0$ is a regular free boundary point and (a) holds.  

We now show that (a) implies (c). Assume that (a) holds, so $x_0$ is a regular free boundary point. By Lemma \ref{higher_dimensional_free_boundary_Lipschitz_graph}, in some systems of coordinates, $\partial S_i \cap \Omega$ is locally the graph of a function $\psi$, $\psi(0,\cdots,0) = 0$ and $\psi$ is differentiable at the origin with $D\psi(0,\cdots,0) = (0,\cdots,0)$. By Lemma \ref{density_lemma} and the fact that the free boundary $\partial\{u_i > 0\} \cap \Omega$ satisfies an exterior sphere condition, it follows that $\lim_{r \rightarrow 0^+}\frac{|B_r(x_0)\cap S_i|}{|B_r(x_0)|}$ exists and is equal to $\frac{1}{2},$ whence (c). 

We now show that (c) implies (a). Assume that (c) holds, so $\lim_{r \rightarrow 0^+}\frac{|B_r(x_0)\cap S_i|}{|B_r(x_0)|} = \frac{1}{2}.$ To show that $x_0$ is a regular free boundary point, assume by contradiction that $x_0$ is a singular point, so there exist two distinct $y_1, y_2 \in \partial C_i \cap \Omega$ with $d(x_0,y_1) = d(x_0,y_2) = \sizeofboundary.$ Note that $\textnormal{angle}(y_1; x_0; y_2) > 0.$ It follows that, for all $r > 0$ small, we have $\frac{|B_r(x_0) \cap S_i|}{|B_r(x_0)|} \leq \frac{|B_r(x_0)\setminus(B_{\sizeofboundary}(y_1) \cup B_{\sizeofboundary}(y_2))|}{|B_r(x_0)|}.$ Hence, $\limsup_{r \rightarrow 0^+} \frac{|B_r(x_0) \cap \{u_i > 0\}|}{|B_r(x_0)|} < \frac{1}{2}.$ This contradicts our assumption. Therefore, (a) holds. 

The equivalence of (d) and (e) follows from the fact that (a) and (b) are equivalent. We show that (d) implies (f). Assume that (d) holds,  so there exist two distinct $y_1, y_2 \in \partial C_i \cap \Omega$ with $d(x_0,y_1) = d(x_0,y_2) = \sizeofboundary.$ Note that $\textnormal{angle}(y_1; x_0; y_2) > 0.$ It follows that for all $r > 0$ small, we have that $\frac{|B_r(x_0) \cap S_i|}{|B_r(x_0)|} \leq \frac{|B_r(x_0)-(B_{\sizeofboundary}(y_1) \cup B_{\sizeofboundary}(y_2))|}{|B_r(x_0)|}.$ Hence, $\limsup_{r \rightarrow 0^+} \frac{|B_r(x_0) \cap S_i|}{|B_r(x_0)|} < \frac{1}{2}$.

We show that (f) implies (d). Assume that (f) holds. If $x_0$ is regular, then by the equivalence of (a) and (c), we have $\lim_{r \rightarrow 0^+}\frac{|B_r(x_0)\cap S_i|}{|B_r(x_0)|} = \frac{1}{2},$ which contradicts to the fact that $\limsup_{r \rightarrow 0^+}\frac{|B_r(x_0)\cap S_i|}{|B_r(x_0)|} < \frac{1}{2}.$ Therefore, $x_0$ is singular and (d) holds. 

In order to see that, for $\mathcal{H}^{n-1}$ a.e. free boundary point $x_0 \in \partial \{u_i > 0\} \cap \Omega,$ the density of $S_i$ at $x_0$ exists, note that by \cite[Corollary 6.5]{CL2}, the set $\{u_i > 0\} \cap \Omega$ is of finite perimeter. We have $\partial \{u_i > 0\} \cap \Omega = \Gamma^{\textnormal{cusp}}_i \cup \Gamma^{\textnormal{reg}}_i \cup ((\partial \{u_i > 0\} \cap \Omega)-(\Gamma^{\textnormal{cusp}}_i \cup \Gamma^{\textnormal{reg}}_i))$. The free boundary points on $\Gamma^{\textnormal{cusp}}_i$ have density zero, and the free boundary points on $\Gamma^{\textnormal{reg}}_i$ have density $\frac{1}{2}.$ Also, the set $((\partial \{u_i > 0\} \cap \Omega)-(\Gamma^{\textnormal{cusp}}_i \cup \Gamma^{\textnormal{reg}}_i)) \subset \partial^e E - \partial^* E,$ with $E = \{u_i > 0\} \cap \Omega$ (recall that $\partial^e E$ is the measure theoretic boundary and $\partial^* E$ is the reduced boundary). Since $\mathcal{H}^{n-1}(\partial^e E - \partial^* E) = 0$ by Federer's Theorem (see Theorem \ref{Federer_Theorem}), the result follows.   
}
\end{proof}
We have the following decomposition of the free boundary:       
\begin{remark}
\label{the_structure_of_the_free_boundary_in_higher_dimensions}
The free boundary set $\partial \{u_i > 0\} \cap \Omega$ can also be written as 
\begin{equation}
\label{structure}
\left(\bigcup_{i \in \mathbb{N}} C_i\right) \cup \Gamma^{\textnormal{cusp}}_i \cup Z,
\end{equation}
where the compact sets $C_i$ are as in Theorem \ref{structure_of_the_reduced_boundary} and $\mathcal{H}^{n-1}(Z) = 0$. This follows from the structure theorem of the reduced boundary with $E := \{u_i > 0\} \cap \Omega$ stated in Theorem \ref{structure_of_the_reduced_boundary}; that is, $\partial ^*E \cap \Omega = \left(\bigcup_{i \in \mathbb{N}} C_i\right) \cup N$
and $\mathcal{H}^{n-1}( N)=0$. Setting $Z := N \cup ((\partial^e E-\partial^* E)\cap\Omega)$ we have that $\mathcal{H}^{n-1}(Z) = 0$ and hence \eqref{structure} follows. 
However, unlike the two-dimensional case, for higher dimensions ($n \geq 3$) it remains open if the set of all cusps has $\mathcal{H}^{n-1}$ measure zero.
$\Box$
\end{remark}

\section{Structure of the Free Boundary}  
For $n = 2$, it was proven in \cite[Lemma 8.9]{CL2} that all singular points are isolated and that, if $\partial S_i \cap \Omega$ consists of only regular points, then $\partial S_i \cap \Omega$ is locally a $C^1$ curve. For $n > 2$, in this section we show that, if the angles at the singular points are bounded away from $\frac{n\omega_n}{2}$, the regular set is open in the free boundary. We obtain a structure result for the free boundary (Theorem \ref{Combined_Results_on_Regular_Set}) and we study the configuration when all supports are convex sets (Theorem \ref{Convex}). We start by splitting the set of singular points in two categories. 

\begin{definition}
\label{twotypes}
A singular free boundary point $x_0 \in \Gamma^{\textnormal{sin}}_i$ is said to be of type 1 if there exist $j, k \in \{1,...,K\} - \{i\}, j \neq k$ such that $d(x_0,y_1) = d(x_0,y_2)$ for some $y_1 \in \partial S_j \cap \Omega$ and $y_2 \in \partial S_k \cap \Omega.$ A singular free boundary point $x_0 \in \Gamma^{\textnormal{sin}}_i$ is said to be of type 2 if it is not of type 1. We denote by $\Gamma^{\textnormal{sin}}_{i,1}$ the set of all singular free boundary points in $\Gamma^{\textnormal{sin}}_i$ that are of type 1, and by $\Gamma^{\textnormal{sin}}_{i,2}$ the set of all singular free boundary points in $\Gamma^{\textnormal{sin}}_i$ that are of type 2.
\end{definition}

By the above definition, a singular free boundary point is of type 1 if it realizes the distance to two different distinct populations, and is of type 2 otherwise. Note that, if $K = 2,$ all singular free boundary points are of type 2. We have the following:
\begin{remark}
\label{remark_on_type_I_singular_points}
If $x_0 \in \Gamma^{\textnormal{sin}}_{i,1}$ is a type 1 singular point, $\textnormal{angle}(x_0) \leq \frac{n\omega_n}{3}$. To see this, note that there exist $j, k \in \{1,...,K\} - \{i\}, j \neq k$ such that $d(x_0,y_1) = d(x_0,y_2)$ for some $y_1 \in \partial S_j \cap \Omega$ and $y_2 \in \partial S_k \cap \Omega.$ Since different populations are at distance exactly $\sizeofboundary$ from each other, we have $d(y_1,y_2) \geq \sizeofboundary$. Recalling \eqref{equation_for_computing_angles} for computing the angles, we have that 
\begin{eqnarray}
\textnormal{angle}(x_0) &:=& \frac{\mathcal{H}^{n-1}(\bigcap_{y \in \partial C_i \cap \Omega, d(x_0,y) = \sizeofboundary} S^+(x_0;y))}{\sizeofboundary^{n-1}} \nonumber \\
&\leq& \frac{\mathcal{H}^{n-1}(S^+(x_0;y_1) \cap S^+(x_0;y_2))}{\sizeofboundary^{n-1}} 
\leq \frac{\frac{1}{3}n\omega_n\sizeofboundary^{n-1}}{\sizeofboundary^{n-1}} 
= \frac{n\omega_n}{3} \nonumber
\end{eqnarray}
\end{remark}

In fact, one can show that if the type 2 singular points form a closed set, the regular set $\Gamma^{\textnormal{reg}}_i$ is open in the free boundary set $\partial S_i \cap \Omega.$      
\begin{lemma}
\label{Partial_Result_Openness_of_Regular_Set}
In arbitrary dimension, let $H$ be given by \eqref{H1} or \eqref{H2}. Then we have that $\Gamma^{\textnormal{sin}}_{i,1}$ is closed in $\partial S_i \cap \Omega.$ Hence, if the set of type 2 singular free boundary points on $\partial S_i \cap \Omega$ is closed, we have that the regular set $\Gamma^{\textnormal{reg}}_i$ is open in the free boundary.
\end{lemma}
\begin{proof}  
For each $m \in \{1,...,K\}-\{i\},$ we define $E_m := \{p \in \partial S_i \cap \Omega: d(p,\partial S_m \cap \Omega) = \sizeofboundary\}.$ Note that each $E_m$ is closed in $\partial S_i \cap \Omega$ as it is the inverse image of the set $\{\sizeofboundary\}$ under the distance function. Now $\Gamma^{\textnormal{sin}}_{i,1} = \bigcup_{j, k \in \{1,...,K\}-\{i\}, j \neq k} (E_j \cap E_k),$ which is closed in $\partial S_i \cap \Omega$. The second statement follows from the first and the relation $\Gamma^{\textnormal{reg}}_i = (\partial S_i \cap \Omega)\setminus(\Gamma^{\textnormal{sin}}_{i,1} \cup \Gamma^{\textnormal{sin}}_{i,2}).$
\end{proof}

\begin{remark}
In general it is unknown if the set $\Gamma^{\textnormal{sin}}_{i,2}$ of type 2 singular free boundary points is closed in the free boundary set $\partial S_i \cap \Omega.$ 
\end{remark}
Although the above question remains unknown in generality, we have $F^{\textnormal{sin}}_i$ is a $F_{\sigma}$ set.

\begin{lemma}
\label{Type_II_Singular_Set}
In arbitrary dimension, let $H$ be given by \eqref{H1} or \eqref{H2}. Then we have the following:
\begin{enumerate}
\item $\Gamma^{\textnormal{sin}}_i$ is an $F_{\sigma}$ set, and the regular set $\Gamma^{\textnormal{reg}}_i$ is $G_{\delta}$.     
\item If $\sup\{\textnormal{angle}(x_0): x_0 \in \Gamma^{\textnormal{sin}}_i\} < \frac{n\omega_n}{2}$, the regular set is open and the singular set is closed.    
\end{enumerate}
\end{lemma}   
\begin{proof}   
1. For any $\eta > 0$ we consider the set $S_{\eta} := \{x_0 \in \Gamma^{\textnormal{sin}}_i: d(x_0,y_1)=d(x_0,y_2)=\sizeofboundary \textnormal{ for some } y_1, y_2 \in \partial C_i \cap \Omega, d(y_1,y_2) \geq \eta\}.$ 
We claim that $S_{\eta}$ is closed for all $\eta > 0$. To see this, assume $\{x_n\}_{n \in \mathbb{N}} \subset S_{\eta}$ and $x_n \rightarrow q$ as $n \rightarrow \infty$. For each $n$ we pick $y_n, z_n \in \partial C_i \cap \Omega$ with $d(y_n,x_n) = d(z_n,x_n) = \sizeofboundary$ and $d(y_n,z_n) \geq \eta$. 
As both $\{y_n\}_{n \in \mathbb{N}}$ and $\{z_n\}_{n \in \mathbb{N}}$ are bounded, by passing to a subsequence if necessary, $y_n \rightarrow y$ and $z_n \rightarrow z$. It follows that $d(y,z) \geq \eta$ and $d(q,y)=d(q,z)=\sizeofboundary$. Hence $q \in S_{\eta}$ and so the set $S_{\eta}$ is closed.
Now note that $\Gamma^{\textnormal{sin}}_i = \bigcup_{k \in \mathbb{N}}S_{\frac{1}{k}}$, so the set $\Gamma^{\textnormal{sin}}_i$ is $F_{\sigma}$. The second statement follows by taking complement.

2. Under this assumption, there exists $\eta > 0$ with $\Gamma^{\textnormal{sin}}_i = S_{\eta}$, which is closed by 1. Hence the regular set is open and the singular set is closed. 
\end{proof}

We are ready to prove a structure Theorem for the regular set. The proof presented here is motivated by Lemma 8.3 in \cite{CL2}.      
\begin{lemma}
\label{main theorem8}
In arbitrary dimension, let $H$ be given by \eqref{H1} or \eqref{H2}. 
Assume that $U \subset \Omega$ is an open set such that all the points on the set $\partial S_i \cap U$ are regular. Then $\partial S_i \cap U$ is locally an $(n-1)$-dimensional submanifold of $\mathbb{R}^n$ of class $C^1.$     
\end{lemma}    
\begin{proof} 
Assume that $x_0 \in \partial S_i \cap U.$ By Lemma \ref{higher_dimensional_free_boundary_Lipschitz_graph}, there exists a differentiable function $\psi$ defined over a small open set $U \subset \mathbb{R}^{n-1}$ containing the origin, and a $r > 0,$ such that, in a system of coordinates $(x_1,x_2, ...,x_n),$ $\partial S_i \cap B_r(x_0)$ is the graph of $\psi$ over $U$. Moreover, in these coordinates, we have that $\psi(0,\cdots,0) = 0$ and $D\psi(0,\cdots,0) = (0,\cdots,0)$. By Corollary 6.2 in \cite{CL2} (the semiconvexity property of the free boundary), there is a tangent ball from below, with uniform radius, at any point on the graph of $\psi.$ 
This implies that for any $z \in U$ there exists a $C^2$ function $\phi_{z}$ tangent from below to the graph of $\psi$ at $z$ and such that $|D^2\phi_{z}| \leq C,$ for some $C > 0$ independent of $z.$ Therefore, we have that for any $z, y \in U,$
\begin{eqnarray}
\psi(y) \nonumber
    &\geq& \phi_{z}(y)\\ \nonumber 
    & \geq & \phi_{z}(z) + \Sigma_{|\alpha| = 1}D^{\alpha}\phi_{z}(z)(y-z)^{\alpha} - C|y-z|^2 \\ \nonumber
    & = & \psi(z) +\Sigma_{|\alpha| = 1}D^{\alpha}\psi(z)(y-z)^{\alpha} - C|y-z|^2,  \nonumber
\end{eqnarray}
where in the above inequality, if $\alpha = (\alpha_1,...,\alpha_{n-1}) \in \mathbb{Z}^{n-1}$ is a multi-index and $x = (x_1,...,x_{n-1}) \in \mathbb{R}^{n-1}$ we define $x^{\alpha} := x_1^{\alpha_1}x_2^{\alpha_2}...x_{n-1}^{\alpha_{n-1}}.$

We fix a point $z_0 \in U$ and consider a sequence $\{z_l\}_{l \in \mathbb{N}} \subset U$ converging to $z_0$ as $l \rightarrow \infty$. For each  multi-index $\alpha$ with $|\alpha| = 1,$ let $p_{\alpha}$ be the limit of a convergent subsequence of $\{D^{\alpha}\psi(z_l)\}_{l \in \mathbb{N}}.$ Passing to the limit in $l \rightarrow \infty$ in the inequality 
\begin{eqnarray}
    \psi(y) \geq \psi(z_l) +\Sigma_{|\alpha| = 1}D^{\alpha}\psi(z_l)(y-z_l)^{\alpha} - C|y-z_l|^2 
    \nonumber
\end{eqnarray}
We obtain that 
\begin{eqnarray}
    \psi(y) \geq \psi(z_0) +\Sigma_{|\alpha| = 1}p_{\alpha}(y-z_0)^{\alpha} - C|y-z_0|^2 
    \nonumber
\end{eqnarray}
for any $y \in U.$ Since all first order partial derivatives of $\psi$ exist at $z_0,$ it follows that $p_{\alpha} = D^{\alpha}\psi(z_0)$ for any multi-index $\alpha$ of degree $1.$ Hence, $\psi$ is of class $C^1$.
\end{proof}


{
In \cite{CL2}, the regularity of the free boundary in the two dimensional case relies on the simpler geometry of a two dimensional cone with vertex at a boundary point, which allows to consider harmonic functions (as barrier functions) vanishing on the edges of the cones. It was shown in \cite{CL2} that all singular points are isolated and the regular set is open and of class $C^1$. These techniques break down in higher dimensions. However, by extending the notion of {\it angles} to higher dimensions and by characterizing singular and regular points in terms of densities (see Theorem \ref{Regular_points_Singular_points_in_Higher_Dimensions}), we were able to show the following structure result of the free boundary in higher dimensions.
}
\begin{theorem}
\label{Combined_Results_on_Regular_Set}
In arbitrary dimension, let $H$ be given by \eqref{H1} or \eqref{H2}. Then the free boundary has the following structure
\begin{equation*}
\partial S_i \cap \Omega = \Gamma^{\textnormal{reg}}_i \cup \Gamma^{\textnormal{sin}}_i, \quad  \Gamma^{\textnormal{sin}}_i = \Gamma^{\textnormal{sin}}_{i,1} \cup \Gamma^{\textnormal{sin}}_{i,2},
\end{equation*}
where $\Gamma^{\textnormal{reg}}_i$ is $G_{\delta}$, $\Gamma^{\textnormal{sin}}_i$ is $F_{\sigma}$ and 
$\Gamma^{\textnormal{sin}}_{i,1}$ is closed. 

In addition, if $\sup\{\textnormal{angle}(x_0): x_0 \in \Gamma^{\textnormal{sin}}_i\} < \frac{n\omega_n}{2}$, the regular set $\Gamma^{\textnormal{reg}}_i$ is open and locally an $(n-1)$-dimensional submanifold of $\mathbb{R}^n$ of class $C^1$.
\end{theorem}
\begin{proof}
The proof follows from Lemma \ref{Partial_Result_Openness_of_Regular_Set}, Lemma \ref{Type_II_Singular_Set}, and Lemma \ref{main theorem8} using Definition \ref{twotypes}.
\end{proof}

We proceed to study the case where the supports of the populations are convex.    
As a consequence of Theorem \ref{Combined_Results_on_Regular_Set}, we have the following:          
\begin{lemma}
\label{Main_Theorem_for_Regular_Set}
In arbitrary dimension, let $H$ be given by \eqref{H1} or \eqref{H2}. Suppose that $\{u_i > 0\} \cap \Omega$ is convex for all $i = 1, ..., K.$ If the singular set is nonempty, then  $\sup\{\textnormal{angle}(x_0): x_0 \in \Gamma^{\textnormal{sin}}_i\} \leq \frac{n\omega_n}{3}$ and the regular set $\Gamma^{\textnormal{reg}}_i$ is open in the free boundary. Moreover, $\partial S_i \cap \Omega$ is locally an $(n-1)$-dimensional hyperplane. 
\end{lemma}              
\begin{proof} 
Let $x_0 \in \Gamma^{\textnormal{sin}}_i$ be a singular free boundary point, so there exists distinct $y_1, y_2 \in \partial C_i \cap \Omega$ with $d(y_1,x_0) = d(y_2,x_0) = \sizeofboundary.$ Since the set $\overline{\{u_i > 0\} \cap \Omega}$ is convex for all $i = 1, ..., K.$ This implies that $y_1 \in \partial S_j \cap \Omega$ and $y_2 \in \partial S_k \cap \Omega$ for some $j, k \in \{1,...,K\}-\{i\}, j \neq k$. 
Therefore, $x_0$ is a type 1 singular point. Hence, all the singular free boundary points on $\partial S_i \cap \Omega$ are of type 1 in the sense of Definition \ref{twotypes}. By Remark \ref{remark_on_type_I_singular_points}, we have that all singular points have angle at most $\frac{n\omega_n}{3}$. Therefore, by Theorem \ref{Combined_Results_on_Regular_Set} the regular set $\Gamma^{\textnormal{reg}}_i$ is open in the free boundary $\partial S_i \cap \Omega$ and is locally a $C^1$ manifold of dimension $n-1$.        

In order to see that the regular set is locally an $(n-1)$-dimensional hyperplane, assume $x_0 \in \Gamma^{\textnormal{reg}}_i,$ and pick a $y \in \partial S_j \cap \Omega$ ($j \neq i$) with $d(x_0,y) = \sizeofboundary.$ Hence, $\{x \in \mathbb{R}^n: (x-x_0)\cdot (y-x_0) = 0\}$ is a supporting hyperplane at $x_0$ of $S_i$. Also, $\{x \in \mathbb{R}^n: (x-y)\cdot (x_0-y) = 0\}$ is a supporting hyperplane at $y$ of $S_j$. Note that since $x_0 \in \Gamma^{\textnormal{reg}}_i$, $d(x_0,y) = \sizeofboundary$ with $y \in \partial S_j \cap \Omega$, we have $d(x_0,\partial S_m\cap\Omega) > \sizeofboundary$ for all $m \neq j, m \neq i$. Combine this with the fact the regular set is open, we can pick a small $r_0 > 0$ such that $B_{r_0}(x_0) \cap (\partial S_i \cap \Omega)$ consists of regular points only and for each $x \in B_{r_0}(x_0) \cap (\partial S_i \cap \Omega)$ we have $d(x,\partial S_j \cap \Omega) = \sizeofboundary$. 

We now assume $z \in B_{r_0}(x_0) \cap (\partial S_i \cap \Omega).$ Recall that $S_i$ is contained in the closed half space $\{x \in \mathbb{R}^n: (x-x_0)\cdot (y-x_0) \leq 0\}$ and $S_j$ is contained in the closed half space $\{x \in \mathbb{R}^n: (x-y)\cdot (x_0-y) \leq 0\}$.  Since the two closed half spaces are at distance $\sizeofboundary$ apart, the only possibility is that $z$ is in the hyperplane $\{x \in \mathbb{R}^n: (x-x_0)\cdot (y-x_0) = 0\}.$ Therefore, all points of $B_{r_0}(x_0) \cap (\partial S_i \cap \Omega)$ lie on the same hyperplane, and the regular set is locally an $(n-1)$-dimensional hyperplane. 
\end{proof}
{
Therefore, for a convex configuration, the free boundary consists of finitely many faces that are $(n-1)$-dimensional hyperplanes. We can also show the following:   
\begin{lemma}
\label{symmetry_result_for_convex_configuration}
In arbitrary dimension, let $H$ be given by \eqref{H1} or \eqref{H2}. Suppose that all $\{u_i > 0\} \cap \Omega$ are convex. Assume $i \neq j$.  
Then we have 
$\mathcal{H}^{n-1}(D_i)=\mathcal{H}^{n-1}(D_j)$, where: 

    $D_i=\{x \in \partial S_i \cap \Omega: d(x,\partial S_j \cap \Omega) = \sizeofboundary\}$ and $D_j= \{x \in \partial S_j \cap \Omega: d(x,\partial S_i \cap \Omega) = \sizeofboundary\}.$
\end{lemma}
\begin{proof}
Note that if $d(S_i,S_j) > \sizeofboundary$, both $D_i$ and $D_j$ are empty, so there is nothing to prove. Therefore, we assume that $d(S_i,S_j) = \sizeofboundary$ in what follows. Since $d(S_i,S_j) = \sizeofboundary$, there exists $x_0 \in \partial S_i \cap \Omega$ and $y_0 \in \partial S_j \cap \Omega$ with $d(x_0,y_0) = \sizeofboundary$. Take $z_0 := \frac{1}{2}(x_0+y_0)$ and consider the hyperplanes $H_{x_0} := \{x \in \mathbb{R}^n: (x-x_0)\cdot(y_0-x_0) = 0\}$, $H_{y_0} := \{x \in \mathbb{R}^n: (x-y_0)\cdot(x_0-y_0) = 0\}$. Since $S_i$ and $S_j$ are convex and are at distance $\sizeofboundary$ from each other, it follows that $H_{x_0}$ is a supporting hyperplane at $x_0$ for $S_i$, and $S_i$ is contained in the closed half space $\{x \in \mathbb{R}^n: (x-x_0)\cdot(y_0-x_0) \leq 0\}$. Similarly, $H_{y_0}$ is a supporting hyperplane at $y_0$ for $S_j$, and $S_j$ is contained in the closed half space $\{x \in \mathbb{R}^n: (x-y_0)\cdot(x_0-y_0) \leq 0\}$. 

Now define $T : \mathbb{R}^n \rightarrow \mathbb{R}^n$ by 
\begin{eqnarray}
    T(x) := x + \frac{2(z_0-x)\cdot(y_0-x_0)}{|y_0-x_0|^2}(y_0-x_0)
\end{eqnarray}
for all $x \in \mathbb{R}^n$ (that is, $T$ is the reflection with respect to the hyperplane $\{x \in \mathbb{R}^n:(y_0-x_0)\cdot(x-z_0) = 0\}$). Clearly $T$ is an affine transformation that preserves the $\mathcal{H}^{n-1}$ measure.         

We claim that $T(D_i) = D_j$. To see this, assume that $b = T(a)$ for some $a \in D_i$, so $d(a,\partial S_j \cap \Omega) = \sizeofboundary$. Since $S_i$ is contained in the closed half space $\{x \in \mathbb{R}^n: (x-x_0)\cdot(y_0-x_0) \leq 0\}$, $S_j$ is contained in the closed half space $\{x \in \mathbb{R}^n: (x-y_0)\cdot(x_0-y_0) \leq 0\}$, and the two closed half spaces are at distance $\sizeofboundary$ apart from each other, the only possibility is that $a \in \partial \{x \in \mathbb{R}^n: (x-x_0)\cdot(y_0-x_0) \leq 0\} = H_{x_0}$. Since $a$ realizes distance $\sizeofboundary$ from $S_j$ and $\emptyset \neq \partial B_{\sizeofboundary}(a) \cap \overline{S_j} \subset \partial B_{\sizeofboundary}(a) \cap \{x \in \mathbb{R}^n: (x-y_0)\cdot(x_0-y_0) \leq 0\} = \{b\}$, we must have that $b \in \partial S_j \cap \Omega$. Hence $b \in D_j$. We have shown that $T(D_i) \subset D_j$. By switching $i$ and $j$ in the above argument we also have $T(D_j) \subset D_i$. We obtain $D_i = T(T(D_i)) \subset T(D_j) \subset D_i$, and so $D_i = T(D_j)$ and $T(D_i) = D_j$. Since $T(D_i) = D_j$ and $T$ is an affine motion, we have $\mathcal{H}^{n-1}(D_i) = \mathcal{H}^{n-1}(D_j)$. 
\end{proof}
}
We give a few remarks regarding Lemma \ref{Main_Theorem_for_Regular_Set} and Lemma \ref{symmetry_result_for_convex_configuration}. 
\begin{remark}
\label{Convex_Polytope}
By Lemma \ref{Main_Theorem_for_Regular_Set} and Lemma \ref{symmetry_result_for_convex_configuration}, it follows that the singular set has Hausdorff dimension at most $n-2$ if all $S_i$ are  convex. 
\end{remark}

\begin{example}
\label{Four_Populations_2D}
Consider the initial configuration of system \eqref{main_problems} as follows: $\Omega := \{(x,y) \in \mathbb{R}^2: x^2+y^2 < 1\}$, $Q_1 := \{(x,y) \in \mathbb{R}^2: x > 0, y > 0\}$, $Q_2 := \{(x,y) \in \mathbb{R}^2: x < 0, y > 0\}$, $Q_3 := \{(x,y) \in \mathbb{R}^2: x < 0, y < 0\}$, $Q_4 := \{(x,y) \in \mathbb{R}^2: x > 0, y < 0\}$. For each $i = 1, \cdots4$, the boundary data $f_i : (\partial \Omega)_{\leq \sizeofboundary} \rightarrow \mathbb{R}$ is defined by requiring that $f_i$ is supported in $(\partial \Omega)_{\leq \sizeofboundary} \cap Q_i$. A possible configuration, after letting $\eps \to 0$, is given by  Figure \ref{figure_for_four_populations_2D}. In fact, assuming uniqueness of the free boundary it follows that this is the unique configuration for the four populations: $S_i = \{(x,y) \in \Omega \cap Q_i: d((x,y),\partial Q_i) > \frac{\sizeofboundary}{2}\}$, and the free boundary $\partial S_i \cap \Omega$ consists of two straight lines meeting at a singular point  with angle $\frac{\pi}{2}$. The uniqueness of the limit configuration is an open problem.


\end{example}

{
\begin{figure}
{
\usetikzlibrary{calc}

\begin{tikzpicture}[scale=3]
    \draw[thick] (0,0) circle (1cm);
    \node at (0.9, 0.9) {$\Omega$};

    \def\hR{0.2} 

    \begin{scope}
        \clip (0,0) circle (1cm);
        
        \filldraw[fill=blue!10, draw=blue!80, thick] (\hR, \hR) rectangle (1.2, 1.2);
        
        \filldraw[fill=green!15, draw=green!80, thick] (-\hR, \hR) rectangle (-1.2, 1.2);
        
        \filldraw[fill=orange!15, draw=orange!80, thick] (-\hR, -\hR) rectangle (-1.2, -1.2);
        
        \filldraw[fill=red!10, draw=red!80, thick] (\hR, -\hR) rectangle (1.2, -1.2);
    \end{scope}

    \fill (\hR, \hR) circle (0.5pt); 
    \fill (-\hR, \hR) circle (0.5pt);
    \fill (-\hR, -\hR) circle (0.5pt);
    \fill (\hR, -\hR) circle (0.5pt);

    \node[blue!80, font=\footnotesize] at (0.55, 0.55) {$S_1$};
    \node[green!80, font=\footnotesize] at (-0.55, 0.55) {$S_2$};
    \node[orange!80, font=\footnotesize] at (-0.55, -0.55) {$S_3$};
    \node[red!80, font=\footnotesize] at (0.55, -0.55) {$S_4$};

    \draw[<->, dashed, gray] (-\hR, 0.4) -- (\hR, 0.4) node[midway, above, font=\tiny, text=black] {$R$};
    \draw[<->, dashed, gray] (0.4, -\hR) -- (0.4, \hR) node[midway, right, font=\tiny, text=black] {$R$};
\end{tikzpicture}

}
\caption{This figure illustrates Example \ref{Four_Populations_2D}. Each population has a singular point with angle $\frac{\pi}{2}$.}
\label{figure_for_four_populations_2D}
\end{figure}
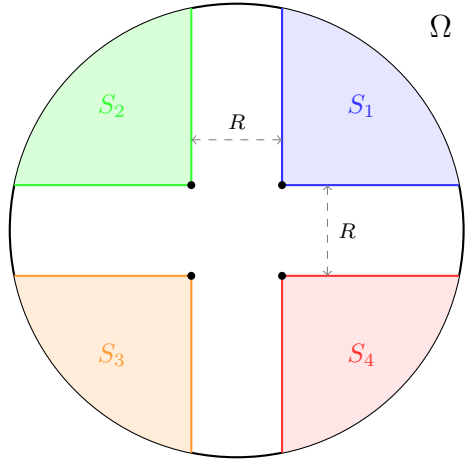
}

 Also, note that in general for a nonempty compact convex set in $\mathbb{R}^n$, $n \geq 2$, the set of all differentiable boundary points (i.e. the boundary points with a unique supporting hyperplane) needs not be open in the topological boundary, as shown in the following example.  
\begin{example}
\label{patological_convex_set}
We give an example of a compact convex set for which the set of differentiable boundary points is not open.
Take $B := \{(x,y) \in \mathbb{R}^2: x^2+y^2 \leq 1\}$ to be the unit compact disk in the plane. For each $k \in \mathbb{N},$ we define $p_k := (\cos \frac{1}{k}, \sin \frac{1}{k}) \in \partial B$. For any $k \in \mathbb{N},$ we consider the closed half space $H_k := \{(x,y) \in \mathbb{R}^2: y \leq \frac{\sin \frac{1}{k+1}-\sin \frac{1}{k}}{\cos \frac{1}{k+1}-\cos \frac{1}{k}}(x-\cos \frac{1}{k})+\sin \frac{1}{k}\},$ and take $A := \bigcap _{k \in \mathbb{N}} (B \cap H_k)$. It follows that $A$ is a nonempty compact convex set in the plane. Also, each $p_k \in \partial A$ is a point of non-differentiability. However, $p_k \rightarrow (1,0)$ as $k \rightarrow \infty,$ and $(1,0) \in \partial A$ is a point of differentiability. 
This shows that the set of all differentiable boundary points is not open in the boundary. 
On the other hand, our results show that these types of sets can not be the supports of the populations $u_i$.
\end{example}
In the two dimensional case, it was shown in (\cite{CL2} Theorem 8.10) that if $x_0 \in \partial S_i \cap \Omega$ and $y_0 \in \partial S_j \cap \Omega$ are free boundary points ($j \neq i$) with $d(x_0,y_0) = \sizeofboundary$, then $\textnormal{angle}(x_0) = \textnormal{angle}(y_0)$. For general dimensions, in Corollary \ref{Weak_Form_Equality_of_Angles}, we were able to show a partial version of this result for convex supports.

\begin{corollary}
\label{Weak_Form_Equality_of_Angles}
In arbitrary dimension, let $H$ be given by \eqref{H1} or \eqref{H2}. If we assume moreover that $\{u_i > 0\} \cap \Omega$ is convex for all $i = 1, \cdots,K$, then for each free boundary point $x_0 \in \partial S_i \cap \Omega$ the density $\lim_{r \rightarrow 0^+}\frac{|B_r(x_0)\cap S_i|}{|B_r(x_0)|}$ exists. In addition, if $x_0 \in \partial S_i \cap \Omega$ and $y_0 \in \partial C_i \cap \Omega$ are free boundary points with $d(x_0,y_0) = \sizeofboundary$, we have that $x_0$ is regular if and only if $y_0$ is regular.
\end{corollary} 
\begin{proof}
By Lemma \ref{Main_Theorem_for_Regular_Set} and Remark \ref{Convex_Polytope}, the regular set is locally a hyperplane and $\{u_i > 0\} \cap \Omega$ is of the form $E_i \cap \Omega$, where $E_i$ is a convex polytope (a finite intersection of open half spaces). Hence, for each free boundary point $x_0 \in \partial S_i \cap \Omega$ the density $\lim_{r \rightarrow 0^+}\frac{|B_r(x_0)\cap S_i|}{|B_r(x_0)|}$ exists. Now assume that $x_0 \in \partial S_i \cap \Omega$ and $y_0 \in \partial C_i \cap \Omega$ are free boundary points with $d(x_0,y_0) = \sizeofboundary$. We may assume $y_0 \in \partial S_j \cap \Omega$ for some $j \neq i$. If $x_0$ is regular, $\partial S_i \cap \Omega$ is locally an $(n-1)$-dimensional hyperplane by Lemma \ref{Main_Theorem_for_Regular_Set}. Since distance $\sizeofboundary$ is always realized from $\partial S_i \cap \Omega$ by $\partial S_j \cap \Omega$, $\partial S_j \cap \Omega$ is also locally a hyperplane near $y_0$. Therefore, at $y_0$ the density $\lim_{r \rightarrow 0^+}\frac{|B_r(y_0)\cap S_j|}{|B_r(y_0)|} = \frac{1}{2}$. By Theorem \ref{Regular_points_Singular_points_in_Higher_Dimensions}, $y_0$ is regular.  
\end{proof}



\begin{theorem}
\label{Convex}
In arbitrary dimension, let $H$ be given by \eqref{H1}  or \eqref{H2}.
Assume that $S_i$ is convex for all $i = 1, ..., K$. Then $\sup\{\textnormal{angle}(x_0): x_0 \in \Gamma^{\textnormal{sin}}_i\} \leq \frac{n\omega_n}{3}$  and the regular set is open and locally an $(n-1)$ dimensional hyperplane. In addition, we have the following:

\noindent {\bf 1.} For any $i \neq j$, $\mathcal{H}^{n-1}(D_i)=\mathcal{H}^{n-1}(D_j)$, where \\
    $D_i=\{x \in \partial S_i \cap \Omega: d(x,\partial S_j \cap \Omega) = \sizeofboundary\}$ and $D_j= \{x \in \partial S_j \cap \Omega: d(x,\partial S_i \cap \Omega) = \sizeofboundary\}.$ \\
 \noindent{\bf 2.} Moreover, if $x_0 \in \partial S_i \cap \Omega$ and $y_0 \in \partial C_i \cap \Omega$ are free boundary points with $d(x_0,y_0) = \sizeofboundary$, then $x_0$ is regular if and only if $y_0$ is regular.
\end{theorem}
\begin{proof}
The proof follows by combining Lemma \ref{Main_Theorem_for_Regular_Set}, Lemma \ref{symmetry_result_for_convex_configuration}, and Corollary \ref{Weak_Form_Equality_of_Angles}. 
\end{proof}

\section*{Acknowledgment}
We would like thank  Stefania Patrizi for discussions concerning the regularity of the free boundary in two dimensions established in \cite{CL2}.

\end{document}